\documentclass[11pt,a4paper,oneside]{amsart}
\usepackage{a4wide}
\usepackage[utf8]{inputenc}
\usepackage[english]{babel}

\usepackage{amsfonts,amsmath,amssymb,bbm}
\usepackage{aliascnt}
\usepackage{enumitem}
\setlist[enumerate]{label=\normalfont{(\roman*)}}
\usepackage{csquotes}
\usepackage{microtype}
\usepackage[all,cmtip]{xy}

\usepackage[pdfusetitle,hidelinks]{hyperref}
\usepackage{amsthm}

\newcommand*{\mc}[1]{\mathcal{#1}}
\newcommand*{\opname}[1]{\operatorname{#1}}

\newcommand*{\GL}{\opname{GL}}
\newcommand*{\SL}{\opname{SL}}
\newcommand*{\PGL}{\opname{PGL}}
\renewcommand*{\phi}{\varphi}

\newcommand*{\NN}{\mathbbm{N}}
\newcommand*{\ZZ}{\mathbbm{Z}}
\newcommand*{\CC}{\mathbbm{C}}
\newcommand*{\QQ}{\mathbbm{Q}}
\newcommand*{\RR}{\mathbbm{R}}
\newcommand*{\EE}{\mathbbm{E}}
\newcommand*{\PP}{\mathbbm{P}}
\newcommand*{\HH}{\mathbbm{H}}

\newcommand*{\Jac}{\opname{Jac}}
\newcommand*{\ka}{\kappa}
\newcommand*{\la}{\lambda}
\newcommand*{\ga}{\gamma}

\newcommand*{\rk}{\opname{rk}}
\newcommand*{\pr}{\opname{pr}}
\newcommand*{\id}{\opname{id}}
\newcommand*{\tup}[1]{\underline{#1}}

\newcommand*{\len}{\opname{l}}

\newcommand*{\End}{\opname{End}}

\newcommand*{\Fl}{\opname{Fl}}
\newcommand*{\St}{\opname{St}}
\newcommand*{\norm}[1]{\|#1\|}
\newcommand*{\tr}{\opname{tr}}

\newcommand*{\Spec}{\opname{Spec}}

\newcommand*{\sstable}{\textup{(}semi\nobreakdash-\textup{)}stable}
\newcommand*{\QSw}{\opname{QSw}}
\newcommand*{\QHP}{\opname{QHP}}
\newcommand*{\HP}{\opname{HP}}

\addto\extrasenglish{%
}

\newcounter{internal}[section]
 
\newaliascnt{intthm}{internal} 
\newaliascnt{intprop}{internal} 
\newaliascnt{intcor}{internal} 
 
\newaliascnt{intlemma}{internal} 
\newaliascnt{intdef}{internal} 
\newaliascnt{intex}{internal} 
\newaliascnt{intrem}{internal}

\theoremstyle{plain}
\newtheorem{thm}[intthm]{Theorem}
\newtheorem{prop}[intprop]{Proposition}
\newtheorem{cor}[intcor]{Corollary}
\newtheorem{lemma}[intlemma]{Lemma}
\theoremstyle{definition}
\newtheorem{mydef}[intdef]{Definition}
\theoremstyle{remark}
\newtheorem{rem}[intrem]{Remark}

\begin{document}
\title[Parabolic Higgs Bundles as Asymptotically Stable Decorated Swamps]
  {Stable Parabolic Higgs Bundles as Asymptotically Stable Decorated Swamps}
\author{Nikolai Beck}
\email{Nikolai.Beck@b-tu.de}
\address{
  Mathematisches Institut\\
  BTU Cottbus--Senftenberg\\
  PF 101344\\
  03013 Cottbus\\
  Germany}

\begin{abstract}
Parabolic Higgs bundles can be described in terms of decorated swamps, which we studied in a recent paper. 
This description induces a notion of stability of parabolic Higgs bundles depending on a
parameter, and we construct their moduli space inside the moduli space of decorated swamps. 
We then introduce asymptotic stability of decorated swamps in order to study the behavior 
of the stability condition as one parameter approaches infinity. The main result is
the existence of a constant, such that stability with respect to parameters greater than this constant is equivalent 
to asymptotic stability. This implies boundedness of all decorated swamps which are semistable with respect to some parameter. Finally, we recover the usual stability condition 
of parabolic Higgs bundles as asymptotic stability.
\end{abstract}

\maketitle

\section{Introduction}
Let $X$ be smooth projective curve over the complex numbers. By the famous theorem of Narasimhan--Seshadri  stable vector bundles on $X$ of degree zero correspond to irreducible unitary representations of the fundamental group $\pi_1(X)$ in $\GL(n,\CC)$ \cite{NaSe1965}. In order to describe all irreducible representations one needs to look at stable Higgs bundles, which Hitchin introduced in \cite{hitchin1987}. Their moduli space was then constructed by Nitsure \cite{nitsure1991moduli}. If $x_0$ is a point of $X$, the irreducible unitary representations of $\pi_1(X\setminus\{x_0\})$ with a fixed monodromy are in one-to-one correspondence with the stable parabolic vector bundles on $X$, i.e., vector bundles with a weighted flag in the fiber over $x_0$ \cite{mehta80}. As a combination of these results, there is a bijection between all irreducible representations of $\pi_1(X\setminus \{x_0\})$ and stable parabolic Higgs bundles \cite{simpson1990}. In order to be able to construct a compact moduli space, we consider the larger 
category of 
parabolic Hitchin pairs. This strategy is similar to the compactification of the moduli space of Higgs bundles (without parabolic structure) in \cite{Schmitt1998}.

In a recent article we studied vector bundles with a general global and local decoration, which we called decorated swamps \cite{Beck2014DecSwamps}. More precisely, we introduced a notion of stability, which  depends on two positive rational parameters $\delta_1$ and $\delta_2$, and constructed the moduli space of stable objects. In this paper, we show that parabolic Hitchin pairs can be realized as a subset of decorated swamps. In this way, they inherit a notion of $\delta_1$-stability. 

As a first result we identify the moduli space of $\delta_1$-stable parabolic Hitchin pairs as a closed subscheme of the moduli space of $(\delta_1,\delta_2)$-stable decorated swamps. In \autoref{sec:Asymptotic_stability} we introduce the notion of asymptotic $\delta_2$-stability for decorated swamps to study the limit $\delta_1\to \infty$. As the main result of this paper we prove the existence of a constant $\Delta$, such that asymptotic $\delta_2$-stability is equivalent to $(\delta_1,\delta_2)$-stability for all $\delta_1\ge \Delta$. An immediate consequence is the boundedness of the class of all decorated swamps which are $(\delta_1,\delta_2)$-semistable for any $\delta_1$. In \autoref{sec:stable_PHB} we are finally able to show that asymptotic stability reproduces the usual parameter free stability condition for parabolic Higgs bundles. Hence, the usual moduli space of stable parabolic Hitchin pairs is a closed subscheme of the moduli space of asymptotically stable decorated swamps.

\subsection*{Notation and Conventions}
Throughout this article we will use the notation introduced in \cite{Beck2014DecSwamps}. In particular, we identify a vector bundle $E$ with its sheaf of sections, and we denote by $\PP(E)$ the hyperplane bundle $\opname{Proj}(\opname{Sym}^*E)$

\subsection*{Acknowledgement}
This paper is an improved account of the results obtained in Chapter 7 of the authors PhD thesis \cite{Beck2014}. The author would like to thank his supervisor Alexander Schmitt for his encouragement and support.

\section{Preliminaries}
\label{sec:Preliminaries}
In this section we state some less well known facts of Geometric invariant theory (GIT) and recall the necessary definitions from our last paper \cite{Beck2014DecSwamps}.
\subsection{The Instability One-Parameter Subgroup}
Let us remind the reader of some results regarding the instability one-parameter subgroup due to Kempf and Ramanan--Ramanathan.

Let $\rho:G\times X\to X$ be the action of an affine reductive group $G$ on a scheme $X$ with a linearization in a line bundle $L$. 
Mumford introduced the notion of (semi-)stable points and proved that the good quotient (geometric quotient) of the open set of semistable (stable) points exists (\cite{GIT}, Theorem 1.10). 
In the case that $X$ is projective and $L$ is ample the (semi-)stable points can be identified by the Hilbert--Mumford criterion (\cite{GIT}, Theorem 2.1): Let $\la:\CC^*\to G$ be a 
one-parameter subgroup and $x\in X$ a point. The limit point $x_\infty:=\lim_{t\to\infty} \la(t)\cdot x$ is a fixed point for the $\CC^*$ action. The action on 
the fiber $L_{x_\infty}$ is of the form $t\cdot l=t^{\gamma}l$ for some $\ga\in \ZZ$. If one defines $\mu_\rho(\la,x):=-\ga$, then $x$ is (semi-)stable if and only if 
any non-trivial one-parameter subgroup $\lambda: \CC^*\to G$ satisfies $\mu_\rho(\la,x)(\ge)0$.

Suppose $X=\PP(V)$ and the action is given by a representation $\rho:G\to \GL(V)$. Fix a maximal torus $T\subset G$. Then one can decompose $V$ as a direct sum 
$V=\bigoplus_{\chi\in X^*(T)} V^\chi$ of eigenspaces $V^\chi:=\{v\in V\,|\, \rho(t)\cdot v=\chi(t)v\}$. The \emph{set of states of $\rho$} is the finite set
\[
 \St_T(\rho):=\{\chi\in X^*(T)\,|\,V^\chi\neq \{0\}\}\,,
\]
the set of states of a point $[f]\in \PP(V)$ represented by $f\in V^\vee\setminus\{0\}$ is the subset $\St_T(\rho,[f]):=\{\chi\in X_*(T)\,|\, f_{|V^\chi}\neq 0\}$. With this definition one finds for $\la\in X_*(T)$
\[
 \mu_\rho(\la,x)=-\min\{\langle \chi,\la \rangle\,|\, \chi \in \St_T(\rho,x)\}\,.
\]
We now consider the real vector spaces $X_{* \RR}(T):=X_*(T)\otimes_\ZZ \RR$ and $X^*_\RR(T):=X^*(T)\otimes_\ZZ \RR$. With a one-parameter subgroup $\la$ of $G$ we associate the parabolic subgroup 
\[
 Q_G(\la):=\{ g\in G\,|\, \lim_{t\to \infty}\la(t)\cdot g \cdot \la^{-1}(t) \textnormal{ exists in }G\}\,.
\]
Fix a Borel subgroup $B$ containing $T$ and consider the closure of the Weyl chamber 
\[
 C:=\{\la \in X_*(T)\,|\, Q_G(\la)\supset B\}\,.
\]
and the convex rational polyhedral cone $C_\RR\subset X_{* \RR}(T)$ generated by $C$. Two characters $\chi,\chi'\in X^*(T)$ define the \emph{wall}
\[
 W_{\chi,\chi'}:=\{\la \in C_\RR\,|\, \langle \chi-\chi',\la\rangle =0\}\,.
\]
For any finite set of characters $S$ the walls $W_{\chi,\chi'}$, $\chi,\chi'\in S$, determine a decomposition of $C$ into finitely many locally closed, rational, polyhedral cones $C_i$, $i\in I(S)$. The function
\[
 X_{*\RR}\to \RR\,,\qquad \la \mapsto -\min\{\langle \chi,\la \rangle\,|\, \chi \in S)\}
\]
is then linear on $C_i$, $i\in I(S)$. For any $i\in I(S)$ and any edge of $C_i$ there is a unique primitive integral generator and we let $\Gamma(S)$ denote the set of all of these generators. The observations imply the following Lemma.
\begin{lemma} \label{lem:finitely_many_1PSGs}
 For a point $x\in \PP(V)$ the following conditions are equivalent:
 \begin{enumerate}
  \item The point $x$ is \sstable.
  \item For all $g\in G$ and all $\la\in \Gamma(\St_T(\rho))$ we have $\mu_\rho(\la,g\cdot x)(\ge)0$.
 \end{enumerate}
\end{lemma}

There is an embedding $i:G\to \GL(r)$ for some $r$. Let $D_r$ denote the maximal torus of diagonal matrices in $\GL(r)$. Without loss of generality we may assume $i(T)\subset D_r$. The isomorphism $X_*(D_r)\cong \ZZ^r$ and the standard pairing on $\ZZ^r$ induce a Weyl-invariant scalar product $(-,-)_T$ on $X_{*\RR}(T)$. We denote by $\norm{-}_T$ the corresponding norm on $X_{*\RR}(T)$. If $T'$ is another torus and $\la \in X_{*\RR}(T')$, then there is $g\in G$ with $gT'g^{-1}=T$ and we set $\norm{\la}:=\norm{g\la g^{-1}}_T$. By Lemma 2.8 in Chapter 2 of \cite{GIT} this is independent of the choice of $g$. For a point $x\in \PP(V)$ we set
\begin{align*}
  \nu_x: X_{*}(G)\setminus\{0\} & \to \RR\\
    \la& \mapsto \frac{\mu_\rho(\la,x)}{\norm{\la}} \,.
\end{align*}
For later application we note:
\begin{lemma}[Ramanan--Ramanathan, \protect{\cite[Lemma 1.1 (i)]{RamRam84}}]
  \label{lem:one_minimum}
 The induced function $\nu_x$ on $X_{*\RR}(T)$ has at most one negative minimum.
\end{lemma}
The main purpose of the function $\nu_x$ is the definition of the instability one-parameter subgroup.
\begin{prop}[Kempf, {\cite[Theorem 2.2]{kempf1978}}]
 \label{prop:Kempf}
 Let $x\in \PP(V)$ be an unstable point.
 \begin{enumerate}
  \item The function $\nu_x$ attains a minimum $m_0<0$ at a point $\la_0\in X_*(G)$.
 \item Let $T$ be a maximal torus and $\la_0$ an indivisible one-parameter subgroup of $T$ such that $\nu_x(\la_0)=m_0$. Then, for every one-parameter subgroup $\la'$ of $T$ with $\nu_x(\la')=m_0$ there is a positive integer $k$ such that $\la'=k\cdot \la_0$.
 \item If $\la_0$, $\la_1$ are two indivisible one-parameter subgroups such that $\nu_x(\la_0)=\nu_x(\la_1)=m_0$, then $Q(\la_0)=Q(\la_1)$ and $\la_0$ and $\la_1$ are conjugate in $Q(\la_1)$.
 \end{enumerate}
\end{prop}
\begin{mydef}
A one-parameter subgroup $\la_0$ of $G$ as in the proposition is called an \emph{instability one-parameter subgroup} for $x$, and $Q(\la_0)$ is the \emph{parabolic instability subgroup} for $x$. 
\end{mydef}
Let $x=[f]=\in\PP(V)$ be an unstable point and $\la_0$ an instability one-parameter subgroup for $x$. Then there are weights $\ga_1<\ldots <\ga_{k+1}$ and 
subspaces $V^i\subset V$ such that $\la_0(t)\cdot v=t^{\ga_i}v$ for $v\in V^i$, $t\in \CC^*$ and $1\le i \le k+1$. We define $V_j:=\bigoplus_{i=1}^j V^i$, $i_0:=\min\{ j\,|\, f_{|V_j}\neq 0\}$ and $f_{i_0}:=f_{|V_{i_0}}$. 
Then $\lim_{t\to \infty}\rho(\la(t),x)=[f_{i_0}\circ \pr_{V_{i_0}}]$ is the limit point of $x$. We denote the induced point $[f_{i_0}]\in \PP(V_{i_0}/V_{i_0-1})$ by $\bar{x}_\infty$.

The group $H:=Q_G(\la_0)/R_u(Q_G(\la_0))$ acts on the space $V_{i_0}/V_{i_0-1}$.
Let $T$ be a maximal torus with $\la_0\in X_*(T)$. There is a unique real character $\chi_0\in X^*_\RR(H)$, such that $\langle \chi_0,\la\rangle=(\la_0,\la)_T$ for all $\la\in X_{*\RR}(T)$. This definition is in fact independent  of the choice of $T$. We set $k:=\norm{\la_0}^2\in \NN$ and $\chi_*:=\norm{\la_0}m_0\chi_0\in X^*(H)$.

\begin{prop}[Ramanan--Ramanathan, {\cite[Proposition 1.12]{RamRam84}}]
\label{prop:RamRam}
 Let $x\in \PP(V)$ be an unstable point. Then, the limit point $\bar{x}_\infty\in \PP(V_{i_0}/V_{i_0-1})$ is semistable with respect to the linearization in $\mc{O}_{\PP(V_{i_0}/V_{i_0-1})}(k)$ twisted by $\chi_*$.
\end{prop}

\subsection{Stability in the Product Space}
Let $G$ be an affine reductive group and $\rho:G\to \GL(V)$ and $\sigma:G\to \GL(W)$ two representations. The aim of this section is to study the stability of a point in the product $\PP(V)\times \PP(W)$. 
\begin{mydef}
 A point $(x,y)\in\PP(V)\times\PP(W)$ is called \emph{asymptotically \textup{(}semi\nobreakdash-\textup{)}stable} if for any one-parameter subgroup $\la$ of $G$ 
 there is a constant  $M>0$ such that for all $m\ge M$ we have
 \[
  \mu_\rho(\la,x)+ m \mu_\sigma(\la,y)(\ge)0\,.
 \]
\end{mydef}
\begin{rem}
 A point $(x,y)$ is asymptotically (semi-)stable if and only if every one-parameter subgroup $\la$ of $G$ satisfies
 \begin{enumerate}
  \item $\mu_{\sigma}(\la,y) \ge 0$ and
  \item $\mu_{\sigma}(\la,y)=0\; \Longrightarrow \;\mu_{\rho}(\la,x) (\ge) 0$.
 \end{enumerate}
\end{rem}
We first note that one can always twist the linearization so that ordinary stability becomes equivalent to asymptotic stability: 
\begin{lemma}[\protect{\cite[Prop. 2.9]{schmitt2004global}}]
 There is an $n_0\in \NN$ such that for all $n\ge n_0$ a point $(x,y)\in \PP(V)\times\PP(W)$ is \textup{(}semi\nobreakdash-\textup{)}stable with respect to the linearization 
 in $\mc{O}_{\PP(V)}(1)\boxtimes\mc{O}_{\PP(W)}(n)$ if and only if $(x,y)$ is asymptotically \textup{(}semi\nobreakdash-\textup{)}stable.
\end{lemma}

For later purposes we need the following result on instability one-parameter subgroups.
\begin{prop}[Schmitt]
\label{prop:instability_in_product}
 Let $G$ be an affine reductive group and $\rho:G\to \GL(V)$ and $\sigma:G\to \GL(W)$ two representations. Then there is an $n_0\in \NN$ such that for all $n\ge n_0$ and every 
 point $(x,y)\in\PP(V)\times\PP(W)$ which is unstable with respect to the linearization in $\mc{O}_{\PP(V)}(1)\boxtimes\mc{O}_{\PP(W)}(n)$, but for which $y$ is semistable, 
 every instability one-parameter subgroup $\la_0$ for $(x,y)$ satisfies
 \[
  \mu_\sigma(\la_0,y)=0\,.
 \]
\end{prop}
This is Theorem 2.1.10 in \cite{schmitt2004global}. We present a slightly simplified version of the proof.
\begin{proof}
 As before the set of states $\St_T(\rho\otimes\sigma)$ determines a finite set $I$ and a decomposition of $K_\RR$ into rational polyhedral cones $K_i$, $i\in I$, such that 
 for $(x,y)\in \PP(V)\times\PP(W)$ and $i\in I$ there exist characters $\chi_{i,\rho}$, $\chi_{i,\sigma}$ with
 \[
  \mu_\rho(\la,x)=-\langle \chi_{i,\rho},\la\rangle \,,\qquad \mu_\sigma(\la,x)=-\langle \chi_{i,\sigma},\la\rangle
 \]
for all $\la \in K_i$. Without loss of generality we may assume the $K_i$ to be pointed. For each $i\in I$ we choose a hyperplane $H_i$ such that $K_i$ is the cone over the 
polytope $P_i:=K_i\cap H_i$. For an index $i\in I$ let $S_\sigma(i)\subset \St_T(\sigma)$ be the set of states $\chi_\sigma$ such that $\langle \chi_\sigma,\la\rangle \le 0$ 
for all $\la\in K_i$ and $P_i(\chi):=\{\la\in P_i\,|\,\langle\chi_\sigma,\la\rangle=0 \}$ is a proper face of $P_i$. Let $Q_i(\chi_\sigma)\subset P_i$ be the convex hull of 
the vertices of $P_i$ not contained in $P_i(\chi_\sigma)$. Then for every $l\in P_I$ there exist $l_1\in P_i(\chi_\sigma)$, $l_2\in Q_i(\chi_\sigma)$ and $0\le t\le 1$ such 
that $l=(1-t)l_1+tl_2$. For another character $\chi$ we define
\begin{align*}
 N(i,\chi,\chi_\sigma) \colon  \RR \times P_i(\chi_\sigma) \times Q_i(\chi_\sigma) & \to \RR \\
     (t,l_1,l_2) &\mapsto -\frac{\langle\chi, (1-t)l_1+tl_2\rangle}{\norm{(1-t)l_1+tl_2}}\,.
\end{align*}
Because of $\langle \chi_\sigma,l_1\rangle =0$ for $l_1\in P_i(\chi_\sigma)$ and $\langle\chi_\sigma,l_2\rangle<0 $ for $l_2\in Q_i(\chi_\sigma)$ we find
\[
 \frac{\partial}{\partial
t}N(i,\chi_\sigma,\chi_\sigma)(0,l_1,l_2)=-\frac{\langle\chi_\sigma,l_2\rangle\cdot
\norm{l_1}-0}{\norm{l_1}^2}=
  \frac{-\langle\chi_\sigma,l_2\rangle}{\norm{l_1}}>0\,.
\]
Since $P_i(\chi_\sigma)$ and $Q_i(\chi_\sigma)$ are compact, there is an $\epsilon(i,\chi_\sigma)>0$ such that 
\[
\partial_tN(i,\chi_\sigma,\chi_\sigma)(t,l_1,l_2)>0
\]
for all $(t,l_1,l_2)\in R:= [0,\epsilon(i,\chi_\sigma)]\times P_i(\chi_\sigma)\times Q_i(\chi_\sigma)$. 

For $i\in I$, $\chi_\sigma\in S_\sigma(i)$ let $S_\rho(i,\chi_\sigma)\subset \St_T(\rho)$ be the set of states $\chi_\rho$ such that there exists $l \in P_i(\chi_\sigma)$ 
with $\langle\chi_\rho,l\rangle>0$. We define
\begin{align*}
 C_1(i,\chi_\rho,\chi_\sigma) :&= \min\{\partial_t N(i,\chi_\rho,\chi_\sigma)(t,l_1,l_2)\,|\,
(t,l_1,l_2)\in R\}\,,\\
 C_2(i,\chi_\sigma) :&= \min\{\partial_t N(i,\chi_\sigma,\chi_\sigma)(t,l_1,l_2)\,|\,
(t,l_1,l_2)\in R\}>0\,.
\end{align*}
Now we choose $n>n_0:=\max\{n_1,n_2\}$ with
\begin{align*}
 n_2:&=\max\left\{-\frac{C_1(i,\chi_\rho,\chi_\sigma)}{C_2(i,\chi_\sigma)}\,|\, i\in I,\chi_\sigma\in S_\sigma(i),\chi_\rho\in S_\rho(i,\chi_\sigma) \right\}\,, \\
 n_1:&=\max\{\langle\chi,\lambda\rangle \,|\,\chi \in\St_T(\rho),\lambda \in \Gamma(\St_T(\rho\otimes\sigma))\}\,.
\end{align*}

Let $(x,y)\in \PP(V)\times\PP(W)$ be unstable with respect to $\mc{O}_{\PP(V)}(1)\boxtimes \mc{O}_{\PP(W)}(n)$ and let $\la_0$ be an instability one-parameter subgroup. 
Then, there is an element $g\in G$ such that $\la':=g\la g^{-1}$ lies in $K$. We consider the points $x':=\rho(g,x)$ and $y':=\sigma(g,y)$. Let $j\in I$ be an index with 
$\la'\in K_j$ and set $\chi_1:=\chi_{j,\rho}$, $\chi_2:=\chi_{j,\sigma}$.

Since $y$ is semistable $F:=\{l \in K_j\,|\,\mu_{\sigma}(l,y')=0\}$ is a face of $K_j$.

(i) In case $F=K_j$ we find $\mu_\sigma(\la_0,y)=\mu_\sigma(\la',y')=0$.

(ii) In case $F=\{0\}$ we find
\[
  \mu_\rho(\lambda,x') + n \mu_\sigma(\lambda,y')\ge -n_1 + n > 0\,
 \]
for all integral primitive generators $\la$ of $K_j$. This contradicts the instability of $(x,y)$.

(iii) Suppose that $F$ is a non-trivial proper face of $K_j$, i.e. $\chi_2\in S_\sigma(j)$. There has to be at least one primitive generator $\la$ of an edge of $F$ 
with $\mu_\rho(\la,x)<0$ because $(x,y)$ is unstable. This shows $\chi_1\in S_\rho(j,\chi_2)$. Then, for all $l_1\in P_j(\chi_2)$, $l_2\in Q_j(\chi_2)$ and $0\le t\le \epsilon(j,\chi_2)$ our choice of $n$ implies
\begin{align*}
 \partial_t (N(j,\chi_1,\chi_2)+n N(j,\chi_2,\chi_2))(t,l_0,l_2) 
 \ge C_1(j,\chi_1,\chi_2)+n_2C_2(j,\chi_2)(t,l_0,l_2)>0\,.
\end{align*}
Hence, for $l(t):=(1-t)l_0+tl_2$ the function
\[
 t\mapsto \nu_\rho(l(t),x')+n\nu_\sigma(l(t),y')=N(j,\chi_1,\chi_2)(t,l_0,l_2)+nN(j,\chi_2,\chi_2)(t,l_0,l_2)
\]
is strictly increasing and the function $\nu_\rho+ n\nu_\sigma$ must attain a negative minimum at a point $l_0\in P_j(\chi_2)$. By \autoref{lem:one_minimum} this is 
the global minimum. Because $\la_0$ was assumed to be the instability one-parameter subgroup $\la'$ is a multiple of $l_0$. Hence, $\la'$ lies in $F$ so that $\mu_\sigma(\la_0,y)=\mu_\sigma(\la',y')=0$.
\end{proof}

\subsection{Decorated Swamps}
We recall the definition of a (semi-)stable decorated swamp from \cite{Beck2014DecSwamps}, Section 3: Let $X$ be a smooth projective curve of genus $g$ and fix two homogeneous representations $\rho:\GL(r)\to V_1$ and $\sigma:\GL(r)\to V_2$. 
\begin{mydef}
A \emph{decorated swamp} is a tuple $(E,L,\phi,s)$ where $E$ is a vector bundle $E$ of rank $r$, $L$ is a line bundle on $X$, $\phi:E_\rho\to  L$ is a non-trivial homomorphism and $s$ is a point in $E^\vee_{\sigma|\{x_0\}}$. 
\end{mydef}
Two decorated swamps $(E,L,\phi,s)$ and $(E',L',\phi',s')$ are considered isomorphic if there are isomorphisms $f:E\to E'$, $\psi:L\to L'$ and a number $c\in \CC^*$ with $\phi'\circ f_\rho=\psi\circ \phi$ and $ s\circ f_{\sigma|\{x_0\}}=c\cdot s$. Here $f_\rho:E_\rho\to E'_\rho$ and $f_\sigma:E_\sigma\to E'_\sigma$ are the isomorphisms induced by $f$. The \emph{type} of a decorated swamp $(E,L,\phi,s)$ is the tuple $(\deg(E),\deg(L))$. In the following we will fix integers $d$ and $l$ and only consider decorated swamps of type $(d,l)$. 

Let $(E,L,\phi,s)$ be a decorated swamp. Recall that a weighted flag of a vector bundle $E$ is a flag $E_\bullet$ of $E$ together with weights $\alpha_i\in \QQ_{>0}$, $1\le i \le \len(E_\bullet)$. We define the function
\[
 M(E_\bullet,\tup{\alpha}):=\sum_{j=1}^{\len(E_\bullet)}\alpha_j\left(\deg(E)\rk(E_j)-\deg(E_j)\rk(E) \right)\,.
\]

By \S 2.4 in \cite{Beck2014DecSwamps}, a weighted flag $(E_\bullet,\tup{\alpha})$ of $E$ induces weighted flags $(E_{\bullet,\rho},\tup{\alpha}_\rho)$ and $(E_{\bullet,\sigma},\tup{\alpha}_\sigma)$ of the associated bundles $E_\rho$ and $E_\sigma$. We restrict these to the generic point $\eta$ of $X$ and the point $x_0$ respectively. Using the notation of \cite[\S 2.1]{Beck2014DecSwamps} we set
\[
 \mu_1(E_\bullet,\tup{\alpha},\phi):=\mu(E_{\bullet,\rho|\eta},\tup{\alpha}_\rho,[\phi])\,, \qquad
 \mu_2(E_\bullet,\tup{\alpha},s):=\mu(E_{\bullet,\sigma|\{x_0\}},\tup{\alpha}_\sigma,[s])\,.
\]
Here, $[\phi]\in \PP(E_{\rho|\eta})$ and $[s]\in\PP(E_{\sigma|\{x_0\}})$ are the points defined by $\phi$ and $s$.
\begin{mydef}\label{def:swamp_stability}
 Let $\delta_1,\delta_2$ be positive rational numbers. We call a decorated swamp $(E,L,\phi,s)$ \emph{$(\delta_1,\delta_2)$-\textup{(}semi\nobreakdash-\textup{)}stable} if the condition
 \[
  M(E_\bullet,\tup{\alpha})+\delta_1\mu_1(E_\bullet,\tup{\alpha},\phi)+\delta_2 \mu_2(E_\bullet,\tup{\alpha},s) (\ge) 0
 \]
holds for all weighted flags $(E_\bullet,\tup{\alpha})$ of $E$.
\end{mydef}

\section{Parabolic Higgs Bundles as Decorated Swamps}
\label{sec:par_Higgs}
In this section, we define $\delta_1$-(semi-)stable parabolic Higgs bundles and construct their moduli space. Let us fix an integer $d$, a line bundle $L$ on $X$, a sequence $0<r_1<\ldots<r_k<r$ of natural numbers and positive rational numbers $\beta_1,\ldots,\beta_k$ with $\sum_{i=1}^k \beta_i<1$.

\subsection{Parabolic Hitchin Pairs}
A \emph{parabolic Higgs bundle} is a vector bundle $E$ of rank $r$ and degree $d$ together with a twisted endomorphism $\phi:E\to E\otimes L$ and a flag $V_\bullet$ of type $\tup{r}$ in $E_{|\{x_0\}}$ which is $\phi$-invariant, i.e. $\phi_{|\{x_0\}}(V_i)\subset V_i\otimes L_{|\{x_0\}}$. An isomorphism $\psi:(E,\phi,V_\bullet)\to (E',\phi',V'_\bullet)$ of parabolic Higgs bundles is an isomorphism $\psi:E\to E'$ such that $\phi'\circ\psi=(\psi\circ\otimes\id_L)\circ \phi$ 
and $\psi_{|\{x_0\}}(V_i)=V'_i$ for $1\le i\le k$.

In order to obtain a projective moduli space we enlarge the category by allowing \enquote{infinite} endomorphisms (compare \cite{Schmitt1998} or Section 2.3.6 in \cite{Schmitt08}).
\begin{mydef}
 A \emph{parabolic Hitchin pair} is a tuple $(E,\phi,\epsilon,V_\bullet)$ where $(E,\phi,V_\bullet)$ is a parabolic Higgs bundle and $\epsilon$ is a complex number such that $\phi$ is non-trivial or $\epsilon\neq 0$.
\end{mydef}
Two parabolic Hitchin pairs $(E,\phi,\epsilon,V_\bullet)$ and $(E',\phi',\epsilon,V'_\bullet)$ are considered \emph{isomorphic} if there are an isomorphism $\psi:E \to E'$ and a number $c\in \CC^*$ with $\psi_{|\{x_0\}}(V_i)\subset V'_i$ for $1\le i\le k$, $\phi'\circ\psi=c\cdot(\psi\circ\otimes\id_L)\circ \phi$ and $\epsilon'=c\cdot \epsilon$. 

A \emph{family of parabolic Hitchin pairs} parameterized by a scheme $S$ is a tuple 
\[
\mc{F}=(E_S,N_S,\phi_S,\epsilon_S,V_{S\bullet})\,,
\]
 where 
\begin{itemize}
\item $E_S$ is a vector bundle of rank $r$ on $S\times X$, such that for every point $s\in S$ the bundle $E_{S|\{s\}\times X}$ is of degree $d$,
\item $N_S$ is a line bundle on $S$,
\item $\phi_S:E_S\to E_S\otimes \pr_S^*N_S \otimes \pr_X^*L$ is a homomorphism,
\item $\epsilon_S:\mc{O}_S\to N_S$ is a homomorphism, such that for all $s\in S$ we have $\phi_{S|\{s\}\times X}\neq 0$ or $\epsilon_{S|\{s\}}\neq 0$ 
\item and $V_{S\bullet}$ is flag of type $\tup{r}$ in $E_{S|S\times\{x_0\}}$.  
\end{itemize}
Note that such a family defines an isomorphism class of parabolic Hitchin pairs for every point $s\in S$.

Two such families $\mc{F}$ and $\mc{F}'$ over $S$ are \emph{isomorphic} if there are a line bundle $T$ on $S$, an isomorphism $f:E_S\to E'_S\otimes\pr_S^* T$ with $f_{|S\times\{x_0\}}(V_\bullet)=V'_\bullet$ and an isomorphism $h_1:N_S\to N'_S\otimes T$ with $(\phi'_S\otimes\id_{\pr_S^*T})\circ f= (f\otimes \pr_S^*h_1\otimes \pr_X^*\id_L)\circ \phi_S$.

\subsection{The Associated Decorated Swamp}
We now explain how to construct a family of decorated swamps (see Definition. 3.7 in \cite{Beck2014DecSwamps}) from a family of parabolic Hitchin pairs: Let $l$ be an integer such that there are inclusions $\iota_1:L\to \mc{O}_X(l)$ and $\iota_2:\mc{O}_X\to \mc{O}_X(l)$. We consider the representation $\rho:\GL(r)\to\GL(\End(\CC^r)^\vee \oplus \CC)$, where $\GL(r)$ acts by conjugation on $\End(\CC^r)^\vee$ and trivially on $\CC$. Let $(E_S,N_S,\phi_S,\epsilon_S,V_{S\bullet})$ be a parameterized family of parabolic Hitchin pairs. Using the evaluation map $\tr:E_S^\vee\otimes E_S\to \mc{O}_X$ we define $\tilde{\phi}$ as the composition
\[\xymatrixcolsep{3.5em}\xymatrix{
 \tilde{\phi}:\End(E_S)^\vee  \ar[r]^-{\id_{E_S^\vee}\otimes\phi_S} &  E_S^\vee\otimes E_S\otimes \pr_S^*N_S \otimes \pr_X^*L  \ar[rr]^-{\tr\otimes\id_{\pr_S^*N_S}\otimes\pr_X^*\iota_1} & &
 \pr_S^*N_S \otimes \pr_X^*\mc{O}_X(l)\,.
 }
\]
The homomorphism $\epsilon_S$ and $\iota_2$ give a homomorphism 
\[
 \tilde{\epsilon}_S:=\pr_S^*\epsilon_S\otimes\pr_X^*\iota_2:\mc{O}_{S\times X}\to \pr_S^*N_S\otimes \pr_X^*\mc{O}_X(l),.
\]
Combined, these define a non-trivial homomorphism
\[
 (\tilde{\phi}_S,\tilde{\epsilon}_S):E_{S,\rho}\cong \End(E_S)^\vee\oplus\mc{O}_{X\times S}\to \pr_X^*\mc{O}_X(l)\otimes\pr_S^*N_S\,.
\]

The flag variety $\opname{Fl}(\CC^r,\tup{r})$ of flags of type $\tup{r}$ in $\CC^r$ can be embedded in the product of $k$ Grassmannians. Using the Pl\"ucker embeddings and the Segre embedding, we get an embedding in $\PP(V_2)$ with
\[
 V_2:= \left(\bigotimes_{i=1}^k \left(\bigwedge^{r_i} (\CC^r)\right)^{\otimes z\beta_i}\right)^\vee\,.
\]
Here, $z$ is the least common denominator of $\beta_1,\ldots,\beta_k$. Let $\sigma$ be the natural action of $\GL(r)$ on $V_2$. Then, there is an embedding of the flag variety $\Fl(E_{S},\tup{r})$ in $\PP(E_{S\sigma})$, and the flag $V_{S\bullet}$ of $E_{S|S\times\{x_0\}}$ determines a section $f:S\to \PP(E_{S,\sigma|S\times \{x_0\}})$. Let 
\[
s_S:E_{S,\sigma|S\times\{x_0\}}\to M_S:=f^*\mc{O}_{\PP(E_{S,\sigma|S\times \{x_0\}})}(1)
\]
be the induced surjective homomorphism. Then, 
\[
\Psi(E_S,N_s,\phi_S,\epsilon_S,V_{S\bullet}):=(E_S,\mc{O}_X(l),N_S,M_S,(\tilde{\phi}_S,\tilde{\epsilon}_S),s_S)
\]
is a family of decorated swamps of type $(d,l)$.

\begin{rem}
  The map $\Psi$ is compatible with isomorphisms and thus induces a natural transformation from the moduli functor of parabolic Hitchin pairs to the moduli functor of decorated swamps. As we will see in \autoref{prop:QHP_closed_subscheme_of_QSw} it is in fact injective.
\end{rem}
 Via $\Psi$ the category of parabolic Hitchin pairs inherits the notion of stability and S-equivalence from the category of decorated swamps. Set $\delta_2:=1/z$.
\begin{mydef}
We call a parabolic Hitchin pair \emph{$\delta_1$-\textup{(}semi\nobreakdash-\textup{)}stable} if its associated decorated swamp is $(\delta_1,\delta_2)$-(semi-)stable.
We call two parabolic Hitchin pairs \emph{S-equivalent} if their associated decorated swamps are S-equivalent. 
\end{mydef}

\subsection{Parabolic Hitchin Quotients}
By Proposition 4.1 in \cite{Beck2014DecSwamps} the class of vector bundles $E$, such that a $\delta_1$-semistable parabolic Hitchin pair with $E$ as the underlying vector bundle exists, is bounded. Hence, there is a number $n_0$ such that for all $n\ge n_0$ and every $(\delta_1,\delta_2)$-semistable parabolic Hitchin pair $(E,\phi,\epsilon,V_\bullet)$ the bundle $E(n)$ is globally generated and $H^1(E(n))$ vanishes. We fix a complex vector space $Y$ of dimension $p(n):=d+r(n+1-g)$.

\begin{mydef}
 A \emph{family of parabolic Hitchin quotients} parameterized by a scheme $S$ is a tuple $(q_S,N_S,\phi_S,\epsilon_S,V_{S\bullet})$, where $q_S:Y\otimes\pr_X^*\mc{O}_X(-n)\to E_S$ is a vector bundle quotient on $S\times X$, such that $(E_S,N_S,\phi_S,\epsilon_S,V_{S\bullet})$ is a family of parabolic Hitchin pairs on $S\times X$ and
 \[
 \pr_{S*}(q_S\otimes\id_{\pr_X^*\mc{O}_X(n)}): Y\otimes \mc{O}_S \to \pr_{S*}E(n)
 \]
 is an isomorphism.
\end{mydef}
The map $\Psi$ also associates a family of decorated quotient swamps (see Definition 4.2 in \cite{Beck2014DecSwamps}) with a family of parabolic Hitchin quotients. This construction induces a natural transformation between the two moduli functors. If the moduli space of parabolic Hitchin quotients exists, this natural transformation defines a morphism to the fine moduli space of decorated quotient swamps $\QSw$ constructed in Proposition 4.3 in \cite{Beck2014DecSwamps}. The following proposition shows that the moduli space does exist and that this morphism is a closed immersion.
\begin{prop}\label{prop:QHP_closed_subscheme_of_QSw}
 The fine moduli space of parabolic Hitchin quotients $\QHP$ exists as a closed subscheme of the moduli space of decorated quotient swamps $\QSw$.
\end{prop}
The proof consists mainly in constructing in inverse to $\Psi$ on an appropriate set.
\begin{proof}
Let  $\Jac^l$ be the Jacobian of line bundles of degree $l$ on $X$ and choose a Poincar\'e bundle $\mc{L}$.
 Recall that $\QSw$ was constructed as a projective scheme over $\Jac^l$. Let $P_1$ be the fiber of $\QSw$ over the point corresponding to the line bundle $\mc{O}_X(l)$. On $\QSw$ we have the universal family $(\tilde{q},\tilde{\ka},\tilde{N},\tilde{M},\tilde{\psi},\tilde{s})$ of decorated quotient swamps. Consider the homomorphism
 \[
  \psi_1:\End(\tilde{E})^\vee\to \End(\tilde{E})^\vee\oplus \mc{O}_{P_1\times X}\to \pr_X^*\mc{O}_X(l)\otimes\pr_{P_1}^*\tilde{N}\to \pr_X^*(\mc{O}_X(l)/L)\otimes\pr_{P_1}^*\tilde{N}
 \]
and let $P_2\subset P_1$ be the closed subscheme such that $\psi_1$ is trivial on $P_2\times X$ (see Proposition 2.3.5.1 in \cite{Schmitt08}). On $P_2\times X$ we have the homomorphism
\[
 \tilde{\phi}:\End(\tilde{E})^\vee \to \pr_X^*L\otimes \pr_{P_2}^*\tilde{N}\,.
\]
Using the homomorphism $1:\mc{O}_{\QSw\times X}\to \tilde{E}\otimes \tilde{E}^\vee$ we construct
\[\xymatrixcolsep{3em}\xymatrix{
 \phi: \tilde{E} \ar[r]^-{1\otimes \id_{\tilde{E}}} & \tilde{E}\otimes \tilde{E}^\vee \otimes \tilde{E} \ar[r]^-{\id_{\tilde{E}}\otimes\tilde{\phi}} & \tilde{E}\otimes\pr_X^*L\otimes \pr_{\QSw}^*N\,.}
\]
Similarly, we consider the homomorphism
\[\xymatrix{
 \psi_2:\mc{O}_{P_2\times X}\ar[r] & \End(\tilde{E})^\vee\oplus\mc{O}_{P_2\times X}\ar[r] & \pr_X^*(\mc{O}_X(l)/\mc{O}_X)\otimes\pr_{P_2}^*\tilde{N}}
\]
and let $P_3\subset P_2$ be the closed subscheme such that $\psi_2$ is trivial on $P_3\times X$. On $P_3\times X$ we now have the homomorphism $\tilde{\epsilon}:\mc{O}_{P_3\times X}\to \pr_{P_3}^*\tilde{N}$. We define $\epsilon:=\pr_{P_3 *}\tilde{\epsilon}:\mc{O}_{P_3}\to \tilde{N}$.

Finally, let $P_4\subset P_3$ be the closed subscheme such that $\tilde{s}$ defines a flag $V_{\bullet}$ of $\tilde{E}_{|P_4\times \{x_0\}}$. Consider the homomorphisms induced by $\phi$
\[
  \phi_{|P_4\times \{x_0\}} V_i\to \tilde{E}_{|P_4\times \{x_0\}}/V_i\otimes \pr_{P_4}^*\tilde{N},\qquad i=1,\ldots,r\,, 
\]
and let $\QHP\subset P_4$ be the closed subscheme where these homomorphisms are trivial. Then, the family $(\tilde{q},\tilde{N},\phi,\epsilon,V_\bullet)$ on $\QHP$ is a family of parabolic Hitchin quotients. It follows from the construction that it is in fact a universal family.
\end{proof}

\subsection{The Moduli Space of Stable Hitchin Pairs}
There is a natural $\PGL(Y)$ action on $\QSw$ and the subscheme $\QHP\subset \QSw$ is $\PGL(Y)$-invariant.
If we choose $n$ sufficiently large, then there is an open subscheme $\QSw^{(\delta_1,\delta_2)\textnormal{-(s)s}}\subset \QSw$ parameterizing families of $(\delta_1,\delta_2)$-(semi\nobreakdash-)stable decorated swamps (see Corollary 5.10 in \cite{Beck2014DecSwamps}). We set $\QHP^{\delta_1\textnormal{-(s)s}}:=\QSw^{(\delta_1,\delta_2)\textnormal{-(s)s}}\subset \QSw\cap \QHP$.
\begin{rem}
 The representation $V_2$ is polynomial and homogeneous of degree 
 \[
 a_2:=\sum_{i=1}^k z\beta_i(r-r_i)\,.
 \]
 In general, to be able to apply the results of \cite{Beck2014DecSwamps}, we need to assume $a_2\delta_2<1$. However, as explained in Remark 7.1 in \cite{Beck2014DecSwamps}, in the case of parabolic  bundles we can weaken the condition to $\sum_{i=1}^k \delta_2z\beta_i<1$, which we assumed in the beginning of this section.
\end{rem}

\begin{lemma}
 The family $\mc{F}:=(\tilde{E},\tilde{N},\phi,\epsilon,V_\bullet)$ on $\QHP^{\delta_1\textnormal{-(s)s}}$ has the following properties:
 \begin{enumerate}
\item  $\mc{F}$ satisfies the local universal property for families of $\delta_1$-\textup{(}semi\nobreakdash-\textup{)}stable parabolic Hitchin pairs.
\item  For two morphisms $f_1,f_2:S\to \QHP^{\delta_1\textnormal{-(s)s}}$ the pullbacks of $\mc{F}$ are isomorphic if and only if there exists a morphism $g:S\to \PGL(Y)$ with $g\cdot f_1=f_2$. 
\end{enumerate}
\end{lemma}
\begin{proof}
 This follows immediately from the corresponding properties of $\QSw$ with respect to families of $(\delta_1,\delta_2)$-semistable decorated swamps (see Section 6.1 in \cite{Beck2014DecSwamps}).
\end{proof}
The existence of the good quotient of $\QSw^{(\delta_1,\delta_2)\textnormal{-(s)s}}$ (proof of Theorem 3.9 in \cite{Beck2014DecSwamps}) implies the existence of the good quotient of $\QHP^{\delta_1\textnormal{-(s)s}}$. The general theory of GIT and moduli spaces (as explained, e.g., in Section 2.2 of \cite{Beck2014DecSwamps}) yields the following result:
\begin{thm}\label{thm:Existence_PH_moduli}
The coarse moduli space $\HP^{\delta_1\textnormal{-(s)s}}$ of $\delta_1$-\textup{(}semi\nobreakdash-\textup{)}stable Hitchin pairs exists as a closed subscheme of the coarse \textup{(}projective\textup{)} moduli space of $(\delta_1,\delta_2)$-\textup{(}semi\nobreakdash-\textup{)}stable decorated swamps.
\end{thm}

\begin{rem}
\begin{enumerate}[wide]
 \item Due to the non-linearity of the stability condition, it is difficult to describe the polystable representative of the S-equivalence class of a given semistable parabolic Hitchin pair explicitly.
\item  In general, the moduli space $\HP^{\delta_1\textnormal{-(s)s}}$ does \emph{not} contain the usual moduli space of parabolic Higgs bundles. In fact, our stability condition depends on the parameter $\delta_1$ while the usual stability condition used in \cite{konno1993,Yokogawa1993} has no parameter dependence. This is not surprising, as the stability condition of (non-parabolic) Hitchin pairs was recovered as the asymptotic stability of swamps in Section 3.6 of \cite{Schmitt04}.
\end{enumerate}
\end{rem}

\section{Asymptotic Stability of Decorated Swamps}
\label{sec:Asymptotic_stability}
In this section we introduce the notion of asymptotic stability for decorated swamps and show that for large enough parameter $\delta_1$ this notion coincides with the stability condition given in \autoref{def:swamp_stability}.

\begin{mydef}
 We call a decorated swamp $(E,L,\phi,s)$ \emph{asymptotically
 $\delta_2$-\textup{(}semi\nobreakdash-\textup{)}stable} if for any weighted flag 
 $(E_\bullet,\tup{\alpha})$  there is a number $c_1\in \QQ_{>0}$ such that for 
 all $\delta_1\ge c_1$ the condition
 \[
  M(E_\bullet,\tup{\alpha})+\delta_1\mu_1(E_\bullet,\tup{\alpha},\phi)+\delta_2\mu_2(E_\bullet,\tup{\alpha},s) (\ge) 0
 \]
holds.
\end{mydef}
\begin{rem} \label{rem:asymptotically_stable}
 A decorated swamp $(E,L,\phi,s)$ is asymptotically $\delta_2$-(semi\nobreakdash-)stable is and only if for any weighted flag $(E_\bullet,\tup{\alpha})$ of $E$ we have
 \begin{enumerate}
  \item $\mu_1(E_\bullet,\tup{\alpha},\phi) \ge 0$ and
  \item $\mu_1(E_\bullet,\tup{\alpha},\phi)=0 \Longrightarrow M(E_\bullet,\tup{\alpha})+\delta_2\mu_2(E_\bullet,\tup{\alpha},s)(\ge) 0$.
 \end{enumerate}
\end{rem}
\begin{prop} \label{prop:stable=>asymptotically_stable}
 For given $\delta_2$ there exists $\Delta_1\in \QQ_{>0}$ such that for all $\delta_1\ge \Delta_1$ a $(\delta_1,\delta_2)$-\textup{(}semi\nobreakdash-\textup{)}stable decorated swamp is also asymptotically $\delta_2$-\textup{(}semi\nobreakdash-\textup{)}stable. 
\end{prop}
The proof uses ideas from the proof of Theorem 2.5.5.2 in \cite{Schmitt08}.
\begin{proof}
 For arbitrary $\delta_1$, Condition (ii) in \autoref{rem:asymptotically_stable} follows from $(\delta_1,\delta_2)$-(semi\nobreakdash-)stability. Suppose $(E,L,\phi,s)$ is a $(\delta_1,\delta_2)$-semistable decorated swamp 
 such that Condition (i) is not satisfied. Let $K$ denote the function field of $X$ and $\eta$ the generic point of $X$. We also define $\EE:=E_\eta$ and $\EE_\rho:=E_{\rho,\eta}$. 
 The assumption means that $x:=[\phi_\eta]\in \PP(\EE_\rho)$ is unstable. Let $\Lambda:K^*\to \SL(\EE_\rho)$ be an instability one-parameter subgroup from \autoref{prop:Kempf} and $(\EE_\bullet,\tup{\alpha})$ its associated weighted flag. The flag $\EE_\bullet$ determines a morphism from $\eta$ to the flag variety $\Fl(E,\tup{r})$, where $\tup{r}$ is the type of $\EE_\bullet$. Since $X$ is smooth and 
 projective there is a unique extension $X\to \Fl(E,\tup{r})$, which determines a flag $E_\bullet$ of $E$. By construction this flag satisfies $\mu_0:=\mu_1(E_\bullet,\tup{\alpha},\phi) \le -1$. The flag $E_\bullet$ also induces a flag $F_\bullet$ of $E_\rho$. Let 
$i_0:=\min\{1\le i \le \len(F_\bullet)\,|\,\phi_{F_i}\neq 0\}$. Then $\phi$ induces a non-trivial homomorphism $\bar{\phi}: F_{i_0}/F_{i_0-1}\to L$. This defines a morphism $f_0:X\to \PP(F_{i_0}/F_{i_0-1})$ with
 \[
  f_0^*\mc{O}_{\PP(F_{i_0}/F_{i_0-1})}(1)= L(-D)
 \]
for some effective divisor $D$ on $X$.
 
 Now $[\bar{\phi}_\eta]$ is the limit point $\bar{x}_\infty$. By \autoref{prop:RamRam} this point is semistable with respect to the linearization of the action of $\HH:=Q_{\SL(\EE)}(\Lambda)/R_{\SL(\EE)}(\Lambda)$ in $\mc{O}_{\PP(F_{i_0}/F_{i_0-1})}(k)$ with $k:=\norm{\Lambda}^2$ twisted by $\chi_*:=\mu_0\chi_{\Lambda}$.

Let $W_\bullet$ be a flag of $\CC^r$ of the same type as $E_\bullet$ and choose an open subset $U\subset X$ with a trivialization $\psi:E_{|U}\to \CC^r\otimes \mc{O}_U$ 
such that $\psi(E_i)=W_i\otimes\mc{O}_U$. This induces an isomorphism $\SL(\EE)\cong \SL(r)\times_{\CC}\Spec(K)$. Then there exists a one-parameter subgroup $\la$ of $\SL(r)$ inducing $\Lambda$. 
The trivialization also defines an isomorphism $\HH\cong Q_{\GL(r)}(\la)/\mc{R}_u(Q_{\GL(r)}(\la))\times_\CC \Spec(K)$, that identifies $\mu_0\chi_\la\times_\CC\id_{\Spec(K)}$ with $\chi_*$. Finally, there is a flag $V_{\bullet,1}$ in $V_1$ of the same type as $F_\bullet$, such that $F_{j|U}\cong V_{j,1}\times U$.

Let $Z:=\PP(V_{i_0,1}/V_{i_0-1,1})^{\textnormal{ss}}/\!\!/H$ be the good quotient with respect to the natural linearization in $\mc{O}_{\PP(V_{i_0,1}/V_{i_0-1,1})}(k)$ twisted by $\mu_0\chi_\la$. Then there is a rational morphism
\[
 \pi:\PP(F_{i_0}/F_{i_0-1}) \dashrightarrow Z\,.
\]
By construction the composition $\pi\circ f_0$ is defined at the generic point and hence extends to a morphism $f:X\to Z$. There is an $m\in \NN_{>0}$ such that $\mc{O}_{\PP(V_{i_0,1}/V_{i_0-1,1})}(k)^{\otimes m}$ descends to an ample line bundle $M$ on $Z$. We find
\[
 f^*(M)\cong ( L^{\otimes k}(-kD)\otimes E_{\mu_0\chi_\lambda})^{\otimes m}(-D')
\]
for another effective divisor $D'$. Here $E_{\mu_0\chi_\la}$ is the line bundle associated to $E$ by the character $\mu_0\chi_\la$. Thus $mkl+m\deg(E_{\mu_0\chi_\la})\ge 0$. Since  $M(E_\bullet,\tup{\alpha})=\deg(E_{\chi_\la})$ we find
\[
 M(E_\bullet,\tup{\alpha}) \le -\frac{kl}{\mu_0}\le kl \,.
\]

Since the set of states is finite and two instability one-parameter subgroups of $x$ are conjugate by \autoref{prop:Kempf}, (iii), there are only finitely many conjugacy classes of possible instability one-parameter subgroups. Hence, one can find  constants $C$ and $C_2$ with
\[
 C\ge \norm{\Lambda}^2 \,,\qquad  C_2\ge  \sum_{j=1}^{\len(E_\bullet)}\alpha_j(r-\rk(E_j))\,.
\]
The $(\delta_1,\delta_2)$-semistability now implies
\begin{align*}
  0 & \le
M(E_\bullet,\tup{\alpha})+\delta_1\mu_1(E_\bullet,\tup{\alpha},\phi)+\delta_2\mu_2(E_\bullet,\tup{
\alpha},s)\\
   &\le Cl  -\delta_1  +a_2\delta_2C_2\,. 
\end{align*}
Thus for $\delta_1> \Delta_1:=\max\{0,Cl +a_2\delta_2C_2\}$ Condition (i) must hold.
 \end{proof}
Before we can prove the converse statement, we need to establish boundedness of asymptotically $\delta_2$-semistable decorated swamps.
\begin{prop} \label{prop:asymptotically_stable=>bounded}
 There is a constant $C$ such that an asymptotically
 $\delta_2$-semistable decorated swamp $(E,L,\phi,s)$ 
 of type $(d,l)$ satisfies
 \[
  \mu_{\max}(E) \le \mu(E)+C\,.
 \]
\end{prop}
\begin{proof}
 Let $F\subset E$ be a subbundle. The quotient $E\to E/F$ and $\phi$ determine a morphism
 \[
  f:X\to \opname{Gr}(E,r-\rk(F))\times_X \PP(E_\rho)\to \PP\left(\bigwedge^{r-\rk(F)}E\right)\times_X \PP(E_\rho)\,.
 \]
By \autoref{prop:instability_in_product} there is an $n(r')$ such that for $n\ge n(r')$ and a point $(x,y)\in \PP(\bigwedge^{r-r'}\EE)\times \PP(\EE_\rho)$ which is unstable with respect to $\mc{O}_{\PP(\bigwedge^{r-r'} \EE)\times \PP(\EE_\rho)}(1,n)$, but where $y$ is semistable, any instability one-parameter subgroup $\la_0$ for $(x,y)$ satisfies $\mu(\la_0,y)=0$. We now choose $n:=\max\{n(r')\,|\, 1\le r'\le r\}$.

If $f$ is generically stable we find
\[
 \deg(f^*\mc{O}_{\opname{Gr}(E,r-\rk(F))\times \PP(E_\rho)}(1,n))=\deg(\det(E/F)\otimes L(-D)^{\otimes
n}) \ge 0
\]
for some effective divisor $D$ on $X$. Thus, $\deg(F)\le \deg(E)+nl$.

If $f$ is not generically stable, there is an instability one-parameter subgroup $\la$ of $\SL(\EE)$ inducing a weighted flag $(E_\bullet,\tup{\alpha})$, such that
\[
 \deg(\det(E/F)\otimes L(-D)^{\otimes n}\otimes E_{\mu_0\chi_\la})\ge 0\,.
\]
From this follows
\[
 \deg(E)-\deg(F)+nl +  \mu_0 M(E_\bullet,\tup{\alpha}) \ge 0\,.
\]

By our choice of $n$ we have $\mu_1(E_\bullet,\tup{\alpha},\phi)=0$. Condition (ii) of \autoref{rem:asymptotically_stable} therefore gives
\[
 M(E_\bullet,\tup{\alpha})+\delta_2\mu_2(E_\bullet,\tup{\alpha},s)\ge 0\,.
\]
There is a constant $C'$ such that $C'\ge \mu_2(E_\bullet,\tup{\alpha},s)$ for all instability one-parameter subgroups. This implies
\[
 \deg(F)\le \deg(E)+nl -\mu_0\delta_2C'\,. 
\]
From this one easily deduces the claim.
\end{proof}
We can now prove the central result of this article.
\begin{thm}\label{thm:asymptotic_stability}
 For fixed $\delta_2\in\QQ_{>0}$ there is constant $\Delta\in \QQ_{>0}$ such that for all $\delta_1> \Delta$ a decorated swamp of type $(d,l)$ is $(\delta_1,\delta_2)$-\textup{(}semi\nobreakdash-\textup{)}stable if and only if it is asymptotically $\delta_2$-\textup{(}semi\nobreakdash-\textup{)}stable.
\end{thm}
\begin{proof}
 Let $(E,L,\phi,s)$ be an asymptotically $\delta_2$-(semi-)stable decorated swamp. Note that by \autoref{lem:finitely_many_1PSGs} there is a finite set $T$ of types of weighted flags for which semistability has to be checked. We define
 \[
  C_2:=\min\left\{\sum_{i=1}^{\len(\tup{\alpha})} \alpha_i r_i \biggm| (\tup{r},\tup{\alpha})\in T \right\}\,.
 \]
 Further we let $m\in \NN$ be a number such that $mr\alpha_j$ is an integer for all $(\tup{r},\tup{\alpha})\in T$, $1\le j \le \len(\tup{\alpha})$.
 
 By \autoref{prop:asymptotically_stable=>bounded} there is a constant $C$ with $\mu_{\max}(E)\le \mu(E)+C$. We define
\[
 M_0:=\max\left\{\,\sum_{j=1}^{\len(\tup{\alpha})}\alpha_j\rk(E)r_j C \biggm| (\tup{r},\tup{\alpha})\in T\right\}\,
\]
and assume $\delta_1> -m(M_0+\delta_2a_2C_2)$.

Let now $(E_\bullet,\tup{\alpha})$ be a weighted flag of type $t\in T$. If $\mu_1(E_\bullet,\tup{\alpha},\phi)=0$ holds, then $(\delta_1,\delta_2)$-(semi\nobreakdash-)stability follows directly from Condition (ii) in \autoref{rem:asymptotically_stable}. Otherwise we have $\mu_1(E_\bullet,\tup{\alpha},\phi)\ge 1/m$, so that
\[
 M(E_\bullet,\tup{\alpha})+\delta_1\mu_1(E_\bullet,\tup{\alpha},\phi)+\delta_2\mu_2(E_\bullet,\tup{\alpha},s)\ge -M_0+\frac{\delta_1}{m} -\delta_2a_2C_2 >0\,.
\]
Together with \autoref{prop:stable=>asymptotically_stable} the claim follows for $\delta_1> \max\{\Delta_1,-m(M_0+a_2C_2\delta_2)\}$.
\end{proof}

\begin{prop}
 The class of vector bundles $E$ such that there exist a $\delta_1\in\QQ_{>0}$ and a $(\delta_1,\delta_2)$-\textup{(}semi\nobreakdash-\textup{)}stable decorated swamp $(E,L,\phi,s)$ of type $(d,l)$ is bounded.
\end{prop}
\begin{proof}
 If $(E,L,\phi,s)$ is $(\delta_1,\delta_2)$-semistable for $\delta_1>\Delta$, then by \autoref{prop:stable=>asymptotically_stable} and \autoref{prop:asymptotically_stable=>bounded} we have $\mu_{\max}(E)\le \mu(E)+C$.
 
 Now suppose $\delta_1\le \Delta$. Then the $(\delta_1,\delta_2)$-semistability with respect to the flag $0\subset F\subset E$ and the weight $\tup{\alpha}=(1)$ gives
 \begin{align*}
  0  \le \rk(E)\rk(F)(\mu(E)-\mu(F))+\Delta a_1(r-1)+\delta_2a_2(r-1)\,.
 \end{align*}
The maximal slope is therefore bounded by
\[
 \mu_{\max}(E)\le \mu(E)+(a_1\Delta+a_2\delta_2)\frac{r-1}{r}\,.
\]
By the usual arguments, the upper bound on the maximal slope implies boundedness.
\end{proof}

\section{Stable Parabolic Higgs Bundles as Asymptotically Stable Decorated Swamps}
\label{sec:stable_PHB}
We come back to the setting of \autoref{sec:par_Higgs}. In particular, recall that $\beta_1,\ldots,\beta_k$ are positive rational numbers, $z$ is their least common denominator and $\delta_2=1/z$.
\begin{mydef}
 We call a parabolic Hitchin pair \emph{\textup{(}semi\nobreakdash-\textup{)}stable} if its associated decorated swamp is asymptotically $\delta_2$-(semi\nobreakdash-)stable.
\end{mydef}
As consequence of \autoref{thm:Existence_PH_moduli} and \autoref{thm:asymptotic_stability} one obtains:
\begin{cor}
 The \textup{(}projective\textup{)} moduli space of \sstable{} Hitchin pairs exists as a closed subscheme of the \textup{(}projective\textup{)} moduli space of asymptotically $\delta_2$-\sstable{} decorated swamps.
\end{cor}
It remains to compare our notion of stability with the usual one.
\begin{lemma} \label{lem:mu_1_HP}
 Let $K$ be a field and $\la:K^*\to \SL(r)$ a one-parameter subgroup with associated weighted flag $(W_\bullet,\tup{\alpha})$ of length $m$. For a point $x=[\Phi,E]\in \PP(\End(K^r)^\vee\oplus K)$ one finds
 \begin{enumerate}
 \item $\mu(\la,x)<0 \quad \Longleftrightarrow\quad $ $E=0$ and $\Phi(W_i)\subset W_{i-1}$ for all $i=1,\ldots,m+1$.
  \item $\mu(\la,x)>0 \quad\Longleftrightarrow \quad W_\bullet$ is not $\Phi$-invariant.
 \end{enumerate}
\end{lemma}
\begin{proof}
If $\Phi\neq 0$, the point $[\Phi]\in \PP(\End(K^r)^\vee)$ satisfies
\[
 \mu(\la,[\Phi])=\max \{\gamma_i-\gamma_j\,|\, \Phi(W_j)/(\Phi(W_j)\cap W_{i-1})\neq 0 \}
\]
with 
 \[
  \gamma_i:= \sum_{j=1}^m \alpha_j\dim(W_j)-\sum_{j=i}^m\alpha_jr\,,\qquad i=1,\ldots,m+1\,.
 \]
 Since $\gamma_i<\gamma_{i+1}$ for $i=1,\ldots,m$, we find $\mu(\la,[\Phi])> 0$ if and only if there is an index $j$ such that $\Phi(W_{j})\nsubseteq W_j$, and $\mu(\la,[\Phi])<0$ if and only if $\Phi(W_i)\subset W_{i-1}$ for all $i=1,\ldots,m+1$.
 
If $E\neq 0$, we have $\mu(\la,[E])=0$ for $[E]\in \PP(K)$. In general, this gives
\[
 \mu(\la,[\Phi,E])=\begin{cases}
                    \max\{\mu(\la,[\Phi]),0\}  & E\neq 0 \neq \Phi\,, \\
		     \mu(\la,[\Phi])  & E=0\,,\\
		      0   & \Phi=0\,,
                   \end{cases}
\]
which implies the claim.
\end{proof}

\begin{prop}
 A parabolic Hitchin pair $(E,\phi,\epsilon,V_\bullet)$ is \textup{(}semi\nobreakdash-\textup{)}stable if and only if the following conditions hold:
 \begin{enumerate}
  \item If $\epsilon=0$, then $\phi$ is not nilpotent,
  \item Every non-trivial $\phi$-invariant proper subbundle $F\subset E$ satisfies
   \[ 
    \frac{\opname{pardeg}_{\tup{\beta}}(F)}{\rk(F)}(\le) \frac{\opname{pardeg}_{\tup{\beta}}(E)}{\rk(E)} \,,
   \]
   where
   \[
 \opname{pardeg}_{\tup{\beta}}(F):= \deg(F) + \sum_{i=1}^k \beta_i \dim( F_{|\{x_0\}}\cap V_i)\,.
\]
 \end{enumerate}
\end{prop}
\begin{proof}
Part (i) of \autoref{lem:mu_1_HP} and the definition of $\mu_1$ imply that Condition (i) of \autoref{rem:asymptotically_stable} is equivalent to Condition (i) in the proposition.

By Part (ii) of \autoref{lem:mu_1_HP}, Condition (ii) of \autoref{rem:asymptotically_stable} is satisfied for weighted flags which are not $\phi$-invariant. For a $\phi$-invariant weighted flag $(E_\bullet,\tup{\alpha})$ we need to check
\[
 M(E_\bullet,\tup{\alpha})+ \delta_2 \mu_2(E_\bullet,\tup{\alpha},s) (\ge) 0\,.                                                                                                                                                                 \]
This is the stability condition for parabolic vector bundles. It is linear in $\tup{\alpha}$ and can thus be checked for invariant subbundles, for which it yields the second condition of the proposition (see Section 7.1 in \cite{Beck2014DecSwamps}).
\end{proof}

\begin{rem}
\begin{enumerate}[wide]
\item Given a semistable parabolic Hitchin pair one can construct its Jordan--H\"older filtration. The unique representative of the S-equivalence class is the associated graded object of this filtration.
\item Our notion of stability of parabolic Hitchin pairs, induced by the asymptotic stability of decorated swamps, reproduces the usual stability condition for parabolic Higgs bundles as given in Definition 1.2 in \cite{konno1993} or Definition 1.3 in \cite{Yokogawa1993}.
\item The combination of our results with the techniques used to construct the moduli space of principal Higgs bundles in \S2.7.4 in \cite{Schmitt08} should lead to the moduli space of parabolic principal $G$-Higgs bundles with a reductive structure group $G$.
\end{enumerate}
\end{rem}

  \bibliographystyle{amsplain}
  \bibliography{bibliography}
\end{document}